\documentclass[12pt]{amsart}

\usepackage{latexsym, amssymb, amsmath}

\usepackage{amsfonts, graphicx}
\usepackage{graphicx,color}
\newcommand{\be}{\begin{equation}}
\newcommand{\ee}{\end{equation}}
\newcommand{\beq}{\begin{eqnarray}}
\newcommand{\eeq}{\end{eqnarray}}

\newtheorem{claim}{Claim}[section]

\newtheorem{thm}{Theorem}[section]

\newtheorem{lma}{Lemma}[section]

\newtheorem{cor}{Corollary}[section]
\newtheorem{defn}{Definition}[section]
\theoremstyle{remark}

\numberwithin{equation}{section}
\makeatletter
\newcommand*\owedge{\mathpalette\@owedge\relax}
\newcommand*\@owedge[1]{%
  \mathbin{%
    \ooalign{%
      $#1\m@th\bigcirc$\cr
      \hidewidth$#1\m@th\wedge$\hidewidth\cr
    }%
  }%
}
\makeatother

\def\p{\partial}

\def\p{\partial}
\def\lf{\left}
\def\ri{\right}
\def\e{\epsilon}

\def\Pi{\overline{\displaystyle{\mathbb{II}}}}

\def\heat{\lf(\Delta -\frac{\p}{\p t}\ri)}
\def\K{K\"ahler }

\def\heat{\lf(\frac{\p}{\p t}-\Delta\ri)}

\def\lf{\left}
\def\ri{\right}

\def\e{\epsilon}
\def\p{\partial}

\def\Rm{\operatorname{Rm}}

\def\K{K\"ahler\ }
\def\be{\begin{equation}}
\def\ee{\end{equation}}

\def\bee
{\begin{equation*}}

\def\eee{\end{equation*}}

\def\bee{\begin{equation*}}
\def\eee{\end{equation*}}

\def\e{\epsilon}
\def\lf{\left}
\def\heat{\lf(\frac{\p}{\p t}-\Delta\ri)}

\def\ri{\right}

\def\p{\partial}

\def\K{K\"ahler }

\def\be{\begin{equation}}
\def\ee{\end{equation}}

\def\lf{\left}
\def\ri{\right}

\def\e{\epsilon}

\def\Rm{\text{\rm Rm}}

\def\p{\partial}

\def\p{\partial}

\def\p{\partial}

\def\bee{\begin{equation*}}
\def\eee{\end{equation*}}

\def\e{\epsilon}
\def\lf{\left}
\def\heat{\lf(\frac{\p}{\p t}-\Delta\ri)}

\def\ri{\right}

\def\p{\partial}

\def\K{K\"ahler }

\def\be{\begin{equation}}
\def\ee{\end{equation}}

\def\lf{\left}
\def\ri{\right}

\def\e{\epsilon}

\def\Rm{\text{\rm Rm}}

\def\p{\partial}

\def\p{\partial}

\def\p{\partial}

\def\bee{\begin{equation*}}
\def\eee{\end{equation*}}

\def\e{\epsilon}
\def\lf{\left}
\def\heat{\lf(\frac{\p}{\p t}-\Delta\ri)}

\def\ri{\right}

\def\p{\partial}

\def\K{K\"ahler }

\def\be{\begin{equation}}
\def\ee{\end{equation}}

\def\lf{\left}
\def\ri{\right}

\def\e{\epsilon}

\def\Rm{\text{\rm Rm}}

\def\p{\partial}

\def\p{\partial}

\def\p{\partial}

\usepackage{soul}

\def\bee{\begin{equation*}}
\def\eee{\end{equation*}}

\def\e{\epsilon}
\def\lf{\left}
\def\heat{\lf(\frac{\p}{\p t}-\Delta\ri)}

\def\ri{\right}

\def\p{\partial}

\def\K{K\"ahler }

\def\be{\begin{equation}}
\def\ee{\end{equation}}

\def\lf{\left}
\def\ri{\right}

\def\e{\epsilon}

\def\Rm{\text{\rm Rm}}

\def\p{\partial}

\def\p{\partial}

\def\p{\partial}

\begin{document}           

\title[]{Weakly PIC1 manifolds with maximal volume growth}

\author{Fei He}
\address[Fei He]{School of Mathematical Science, Xiamen University, 422 S. Siming Rd., Xiamen, China 361000}
\email{hefei@xmu.edu.cn}

 \author{Man-Chun Lee}
\address[Man-Chun Lee]{Department of
 Mathematics, University of British Columbia, 121-1984 Mathematics Road, Vancouver, B.C. V6T 1Z2, Canada.}
\email{mclee@math.ubc.ca}

\footnote{F.H. was partially supported by the Fundamental
Research Funds for the Central Universities Grant No. 20720180007.}

\maketitle

\begin{abstract}
In this article we use Ricci flow to show that complete PIC1 manifolds with maximal volume growth are diffeomorphic to $\mathbb{R}^n$. One of the key ingredients is local estimates of curvature lower bounds on an initial time interval of the Ricci flow. As another application of these estimates we obtain pseudolocality type results related to the PIC1 condition.
\end{abstract}

\section{Introduction}
The notion of isotropic curvature was introduced by the seminal work of Micallef and Moore \cite{MicallefMoore1988}. A Riemannian manifold with dimension at least 4 has nonnegative isotropic curvature if $R(\varphi, \bar{\varphi}) \geq 0$ where $\varphi = (e_1+i e_2)\wedge (e_3 + ie_4)$ for any orthonormal 4-frame $\{e_1,e_2,e_3,e_4\}$ and $R$ is the complex linear extension of the Riemannian curvature operator. We say $(M,g)$ is a weakly PIC1 Riemannian manifold if its product with $\mathbb{R}$ has nonnegative isotropic curvature, an algebraic characterization of this curvature condition is given by the following if $M$ has dimension at least 4.
\begin{defn}
The curvature type operator $A$ is weakly $PIC1$ if we have
$$A_{1313}+\lambda^2 A_{1414}+A_{2323}+\lambda^2A_{2424}-2\lambda A_{1234}\geq 0$$
for all $p\in M$, all orthonormal four-frames $\{e_1,e_2,e_3,e_4\}$, all $\lambda \in [0,1]$. For notational convenience, we say that $A\in PIC1$.
\end{defn}
 It is immediate to see a manifold with weakly PIC1 must have  nonnegative Ricci curvature, hence it has at most Eulidean volume growth by the volume comparison theorem. If $M$ has dimension 3 then by simple calculation we see that weakly PIC1 is equivalent to nonnegative Ricci curvature. Complete manifolds with nonnegative Ricci curvature must be covered by $\mathbb{S}^3$, $\mathbb{S}^2\times \mathbb{R}$ or $\mathbb{R}^3$ in dimension 3 \cite{Liu2013}, while in higher dimensions their topology can be much more complicated (see for ex. \cite{Menguy2000}). Weakly PIC1 condition is indeed much stronger in higher dimensions, in this short note we want to show that if such a manifold has maximal volume growth then its topology is trivial.

\begin{thm}\label{main theorem}
Let $(M,g)$ be a complete weakly PIC1 Riemannian manifold with dimension $n \geq 4$, if $(M,g)$ has maximal volume growth, i.e.
\[Vol B(r) \geq v r^n  \]
for some $v >0$ and $\forall r > 0$, then $M$ is diffeomorphic to $\mathbb{R}^n$.
\end{thm}

Clearly $\mathbb{S}^2 \times \mathbb{R}^{n-2}$ with the standard metric is weakly PIC1, so the volume growth assumption in the above theorem cannot be dropped. However, we wonder if it can be weakened to $Vol B(r) > c r^{n-2+\e}$ for some small number $\e>0$. The proof in this note use the tool of Ricci flow on complete manifolds, the maximal volume growth assumption is used to guarantee the long time existence of such a flow without finite time singularity.

In section 2, we first discuss the preservation under complete Ricci flow of certain nonnegative curvature conditions including weakly PIC1, which may be of independent interest. The method is by localizing the maximal principle to obtain lower bound estimates of the curvature, which is the same as in \cite{LeeTam2017}. This method is actually robust enough to work for some other nonnegativity conditions which will not be discussed here.  In order to prove Theorem \ref{main theorem}, we exhibit in section 4 the existence of Ricci flow on noncollapsed complete manifolds with weakly PIC1 by the same method as in \cite{LeeTam2017-2} and \cite{Hochard2016}, we also cite heavily from \cite{CabezasBamlerWilking2017} and \cite{SimonTopping2016}. The short-time existence of Ricci flow in this situation also follows from a very recent work \cite{Lai2018}. Under the assumption of Theorem \ref{main theorem} we can show the solution of Ricci flow exists for all time, has nonnegative Ricci curvature and injectivity radius $\to \infty$, we can then construct a diffeomorphism by an elementary argument detailed in section 5.

In section 3, besides proving necessary curvature estimates for establishing the existence of Ricci flow, we obtain some pseudolocality type results related to PIC1 condition. Recall that Perelman's pseudolocality theorem (\cite{P}) provides an interior curvature estimate for the Ricci flow, given that an initial ball is isoperimetrically close to the Euclidean space and has scalar curvature bounded from below, and the global assumption that the Ricci flow solution is complete with bounded curvature. This result played an important role in the study of complete noncompact Ricci flow. An alternative version of pseudolocality has been obtained in \cite{TianWang2015}, where the initial ball was required to have almost Euclidean volume and almost nonnegative Ricci curvature. In dimension 3, the pseudolocality has been substantially improved by \cite{SimonTopping2016} where the Ricci curvature and volume are only assumed to be bounded from below, hence not necessarily close to the Euclidean space. Using the local preservation estimates in section 2 and the method of \cite{SimonTopping2016}, we formulate a pseudolocality statement assuming PIC1 condition and volume lower bound on the initial ball, see Theorem \ref{pic1pse}. When the initial ball has only PIC1 bounded from below, we need to assume almost Euclidean volume lower bound, see Theorem \ref{pse-almost}. As an application of Theorem \ref{pse-almost}, we obtain Corollary \ref{local-doubling-estimate}, which slightly improves the main theorem in \cite{Lu2010} by relaxing the bounded curvature assumption on the whole space-time to that on positive time slices.

{\it Acknowledgement}: The second author would like to thank his advisor Professor Luen-Fai Tam for his constant support and teaching over years. 

\section{Preservation of curvature conditions}

The main curvature condition of interest in this article is weakly PIC1, however, the method here actually works for some other curvature conditions. Here we present the argument in a more general setting.

Now let $g(t)$ be a solution of Ricci flow. By the trick of Uhlenbeck, we may assume that the curvature $Rm(g(t))$ evolves under
$$(D_t-\Delta) Rm=2Q(Rm)$$
where $Q(R)=R^2+R^{\#}$.

As described in \cite{CabezasBamlerWilking2017}, we regard $Rm$ as a symmetric bilinear form on $\mathfrak{so}(n,\mathbb{R})$ and then extend complex bi-linearly to a map $Rm:\mathfrak{so}(n,\mathbb{C})\times \mathfrak{so}(n,\mathbb{C})\rightarrow \mathbb{C}$. Then we define the convex cone by
$$\mathfrak{C}({S})=\{ Rm\in S^2_B(\mathfrak{so}(n)): Rm(v,\bar v)\geq 0\;\;\text{for all}\;\;v\in S\}$$
where ${S}\subset \mathfrak{so}(n,\mathbb{C})$ is a subset that is invariant under the natural $SO(n,\mathbb{C})$ action and satisfy the following: For any $Rm\in \partial\mathfrak{C}$, $v\in {S}$ with $Rm(v,\bar v)=0$, we have
\begin{align}\label{cone}
Q(Rm)(v,\bar v)\geq 0.
\end{align}

In particular, the PIC1 condition is equivalent to say that $Rm\in \mathfrak{C}({S})$ where
$${S}=\{  v\in\mathfrak{so}(n,\mathbb{C}): \text{rank}(v)=2,\;v^3=0\}$$
which is known to satisfy \eqref{cone} by the work of Brendle and Schoen \cite{BrendleSchoen2009} and Nguyen \cite{Nguyen2010}. We refer interested readers to \cite{Wilking2013} for a more complete list of ${S}$ satisfying $\eqref{cone}$. We also would like to point out that recently Brendle had discovered a new cone \cite{Brendle2018} which satisfies \eqref{cone}.

In the following, we will establish a local estimate on how $Rm_{g(t)}$ fail to stay inside $\mathfrak{C}({S})$ under the assumption $|Rm|\leq at^{-1}$ and hence global preservation of $Rm(g(t))\in\mathfrak{C}$. For simplicity, we will denote $\mathfrak{C}=\mathfrak{C}({S})$ where ${S}$ is a subset in $\mathfrak{so}(n,\mathbb{C})$ such that \eqref{cone} holds.

\begin{thm}\label{pic}
Let $\mathfrak{C}({S})$ be a convex cone satisfying \eqref{cone}. Suppose $(M,g(t))$ is a Ricci flow for $t\in[0,T]$, $a\geq 3,\,\sigma>0$ and $p\in M$ such that $B_t(p,1+4\sigma)\subset\subset M$ for all $t\in [0,T]$. Assume further that
\begin{enumerate}
\item $Rm_{g(0)}\in \mathfrak{C}$ on $B_0(p,1+4\sigma)$,
\item $\displaystyle |Rm(g(t))|\leq \frac{a}{t}$ on $B_t(p,1+4\sigma)$ for all $t\in (0,T]$.
\end{enumerate}
Then there exists $c(n),\,D(n)>0$ such that on $B_t(p,1)$, $k\in \mathbb{N}$,
$$Rm_{g(t)} + t^k \mathcal{I}\in \mathfrak{C}$$
for all $t\leq  T\wedge c(n)\sigma^2 a^{-1}\wedge D\sigma^4k^{-2}$.
Here $\mathcal{I}$ denotes the constant curvature operator of scalar curvature $n(n-1)$.
\end{thm}
\begin{proof}First of all, we will assume $T<1$. The condition that $A\in\mathfrak{C}$ is equivalent to say that $A_{v\bar v}\geq 0$ for all $v\in {S}$ where $A$ is understood to be the complexified operator.

Denote $A=\Phi Rm+ \varphi(t) \mathcal{I}$ where $\varphi(0)>0$ so that $A(0)\in\mathfrak{C}$. Here we denote $\Phi$ to be a generic cutoff function. We will specify our choice of $\varphi$ and $\Phi$ later.

Suppose $A$ fail to be inside $\mathfrak{C}$ for some $t_0>0$, we can then find the largest $t_1>0$ such that $A(t)\in \mathfrak{C}$ for all $t\in [0,t_1]$. At $t=t_1$, there is point $p\in M$, $v\in \mathfrak{so}(n,\mathbb{C})$ with $|v|=1,v\in {S}$ so that $A(v,\bar v)=0$.

Extend $v$ locally using parallel translation with respect to metric $g(t_1)$ and then extend it to spacetime such that $D_t v=0=\Delta v$ at $(p,t_1)$. We are now ready to derive equations regarding to the choice of $\varphi(t)$.

By assumption on $(p,t_1)$, $A_{v\bar v}\geq 0$ locally and hence at $(p,t_1)$
\begin{align*}
0&\geq  (\partial_t -\Delta )A(v,\bar v)\\
&=\Box \Phi \cdot R_{v,\bar v}+2\Phi \cdot Q(Rm)_{v\bar v}-2\langle \nabla\Phi,\nabla R_{v\bar v}\rangle+\varphi' \\
&=-\left(\frac{\Box \Phi}{\Phi}+\frac{2|\nabla \Phi|^2}{\Phi^2} \right)\varphi(t)+\varphi'(t) +2\Phi Q(Rm)_{v\bar v}.
\end{align*}

Here we denote $\heat$ by $\Box$ for notational convenience. By the preservation of cone $\mathfrak{C}$, see \cite{BrendleSchoen2009},
$$Q(A)(v,\bar v)\geq 0.$$
Hence
$$0\leq Q(A)=\Phi^2 Q( Rm)+\varphi \Phi Ric\owedge id -(n-1) \varphi \Phi\cdot  Rm$$
where we have used the fact that at $(p,t_1)$, $A_{v\bar v}=0$ at the last term. Therefore, we have obtained that at $(p,t_1)$,
\begin{equation}\label{inequality regarding to varphi}
\varphi'\leq \left(\frac{\Box \Phi}{\Phi}+2\frac{|\nabla \Phi|^2}{\Phi^2} +|Rm_{v\bar v}|\right)\varphi.
\end{equation}
It reduces back to the method in \K Ricci flow concerning the nonnegativity of $BK$, see \cite{LeeTam2017}, we can use the bootstrapping argument there to deduce the result.

Let $\Phi(x,t)=\phi(d_t(x,p)+c\sqrt{at})$. By Lemma 8.3 in \cite{P}, whenever $d_t(x,p)\geq \sqrt{t}$, it satisfies
$$\heat d_t(x,p)\geq -\beta \sqrt{a} t^{-1/2}$$
in the sense of barrier where $\beta=\beta(n)$. Thus if we choose $c$ sufficiently large, the modified distance function $\tilde d_t(x,p)=d_t(x,p)+c\sqrt{at}$ satisfies
\begin{align}\label{dist-1}
\heat \tilde d_t(x,p)>0
\end{align}
whenever $d_t(x,p)\geq \sqrt{t}$. We may assume it to be smooth when applying maximum principle using the argument in \cite{ST}. Here $\phi$ is some cutoff function identical to $1$ on $[0,1+2\sigma]$, vanishes outside $[0,1+3\sigma]$ so that
$$\heat \Phi \leq -\phi''(\tilde d_t).$$

We first show that the conclusion hold on a slightly larger ball and some time interval where the curvature is bounded by some constant $|Rm|(t) < C_1$. If we allow the time interval to be uncontrolled, we can assume the curvature to be bounded since we are working on a compact subset in $M$. For each $\epsilon > 0$, we choose $\phi$ such that $\sigma \phi^{-1}|\phi''|+\sigma^2\phi^{-2}|\phi'|^2\leq C\e^{-2} \phi^{-\e}$ for some constant $C$. Using
$0=\Phi Rm+\varphi \mathcal{I} \geq -C_1\Phi+\varphi$, the above inequality \ref{inequality regarding to varphi} becomes
$$(\varphi^\e)'\leq C_2,$$
where $C_2$ depends on $C_1$, $\sigma$ and $\e$. Therefore, we can see that if we choose $\varphi= (t+s)^{\alpha}$ with $\alpha<\e^{-1}$, and
\[ 0 < s \leq s_0:= \frac{1}{2}\left(\frac{\e \alpha}{1+C_2(C_1, \sigma, \e)}\right)^\frac{1}{1-\e \alpha},\]
then $t_1 \leq s$ will lead to a contradiction. Hence $\Phi Rm(s) + (2s)^\alpha \mathcal{I}\in\mathfrak{C}$ on the small time interval $[0, s_0]$. In a nut shell, we have shown the following.
\begin{claim}\label{claim}
For any $k\in \mathbb{N}$, there is $\delta>0$ depending on $k$, $\sigma$ and the initial metric such that for all $t\in[0,\delta]$, $\tilde d_t(x,p) \geq 1+2\sigma$,
$$Rm_{g(t)}+(2t)^k\mathcal{I} \in\mathfrak{C}.$$
\end{claim}

Now, we choose $\phi$ so that $\phi=1$ on $[0,1+\sigma]$, vanishes outside $[0,1+2\sigma]$ and satisfies $\sigma \phi^{-1}|\phi''|+\sigma^2\phi^{-2}|\phi'|^2\leq C\e^{-2} \phi^{-\e}$, where $0< \e < \frac{1}{2(1+2a)}$. Let $a < k < \frac{1}{2\e} - 1$ and choose $\varphi(t) = t^k$. Although $\varphi(0)=0$, we can still find a positive $t_1$ as above since Claim \ref{claim}. guarantees that $A(t) \in \mathfrak{C}$ for $t>0$ small enough. We argue in the exact same way as before. The difference is that by the condition $|Rm|(t)\leq \frac{a}{t}$, the inequality \ref{inequality regarding to varphi} becomes
$$\varphi'\leq \frac{a}{t_1}\varphi +\frac{C a^\e}{\e^2 \sigma^2 t_1^\e}\varphi^{1-\e}.$$
By our choice of $\varphi$ this implies
$$t_1 \geq \left(\frac{\e^2 \sigma^2(k-a)}{C a^\e}\right)^\frac{1}{1-k\e-\e}.$$
Note that if we take $\e = \frac{1}{2k+3}$ and $k \geq 2a$, then we have
\[t_1 \geq \frac{D\sigma^4}{k^2},\]
where $D$ is a constant independent of $a$ and $k$.

The conclusion then follows after we shrink the time interval to $[0, \frac{\sigma^2}{C(n)a} \wedge T]$ in order to compare $d_t(x,p)$ and $\tilde d_t(x,p)$.
\end{proof}

By using Theorem \ref{pic}, we can show that the curvature tensor remains in the cone if the curvature of the complete Ricci flow solution satisfies $ad_{g_0}(x,p)^2 t^{-1}$ when $t>0$ for some $a>0$. Before we give a proof, we first recall a lemma which allows us to compare $d_t(x,p)$ and $d_0(x,p)$.

\begin{lma}[The shrinking balls lemma, \cite{SimonTopping2016}]\label{sbl}
Suppose $(M,g(t))$ is a Ricci flow for $t\in [0,T]$. Then there exists a constant $\beta_n\geq 1$ such that the following is true. Suppose $p\in M$ and $B_0(p,r)\subset\subset M$ for some $r>0$, and $Ric_t\leq (n-1)c_0/t$ on $B_0(p,r)$ for each $t\in (0,T]$. Then
$$B_t\left(p,r-\beta\sqrt{c_0 t}\right)\subset B_0(p,r).$$
\end{lma}

Now we are ready to show the preservation of curvature conditions.
\begin{cor}\label{preservation of nonnegativity}
Let $(M,g(t))$ be a complete solution of Ricci flow with
\[|Rm|(x,t) \leq \frac{a d_{g_0}(x,p)^2}{t}\]
for $t\in(0,T]$, where $d_{g_0}(x,p)$ is the distance to a fixed point $p\in M$ with respect to $g(0)$. Suppose $Rm(0) \in \mathfrak{C}$, then $Rm(t) \in \mathfrak{C}$ for $t\in[0,T]$.
\end{cor}
\begin{proof}
Let $$T_{max}=\sup\{ s\in [0,T]: \;\;Rm(g(t))\in\mathfrak{C},\;\;\forall \;t\leq s\}.$$
We show that $T_{max}=T$. Clearly, $T_{max}\geq 0$.

For each $k$ and $\rho> 0$, we claim that $|Rm|(x,t) \leq \frac{4a \rho^2}{t}$ for $x\in B_t(p,\rho)$ and $0<t < T_0$, where $T_0=T\wedge \frac{1}{4a\beta}$ . To see this, let $\bar{r}$ be the least radius such that $B_t(p, \rho) \subset B_0(p, \bar{r})$ for $0\leq t \leq T_0$, by the shinking ball lemma \ref{sbl} we have $\bar{r}\leq 2\rho$, hence the claim.

Now we can apply Theorem \ref{pic} to show
\[Rm + t^k \mathcal{I} \in \mathfrak{C}\]
on $B_t(p, \rho)$, $t< T_1 = T_0\wedge c(n) \wedge D \rho^4 k^{-2}$. In particular, $T_1$ is independent of $\rho$.  Let $\rho \to \infty$, then we can redefine $T_1$ to be independent of $k$. This shows $T_{max} > 0$. If $T_1< T$ we can repeat the argument to show $T_{max}=T$.
\end{proof}

We would like to point out that in Theorem \ref{pic}, the non-negativity at $t=0$ is crucial if we compare the evolution equation with the corresponding ODE solution. Now,  we try to extend the result which allows the lowest eigenvalue to be negative initially. However, we need to assume a stronger curvature assumption as well as the cone. We will consider convex cone with the following condition: $\exists \lambda \geq 0$ such that for all $Rm\in \partial\mathfrak{C}$, $v\in {S}$ with $Rm(v,\bar v)=0$, we have
\begin{align}\label{cone-2}
(Ric\owedge I -\frac{1}{2} scal \cdot \mathcal{I})_{v\bar v}\leq \lambda \sqrt{Q(Rm)_{v\bar v}}.
\end{align}
In particular, the cone of  $PIC$ also satisfies \eqref{cone-2}. For details, we refer to \cite{CabezasBamlerWilking2017}.

\begin{thm}\label{pic-2}
Let $\mathfrak{C}({S})$ be a convex cone satisfying \eqref{cone-2}.  Suppose $(M,g(t))$ is a Ricci flow for $t\in[0,T]$, and $p\in M$ such that $B_t(p,2)\subset\subset M$ for all $t\in [0,T]$. Assume further that there is $a<\frac{1}{3\sqrt{n}}$, $k>0$ such that
\begin{enumerate}
\item $Rm_{g(0)}+k\mathcal{I}\in \mathfrak{C}$ on $B_0(p,2)$,
\item $\displaystyle |Rm(g(t))|\leq \frac{a}{t}$ on $B_t(p,2)$ for all $t\in (0,T]$.
\end{enumerate}
Then there exists $L(n,k,a), \tilde T(n,a,k)>0,\;p(n,a)\in[0,1/2)$ such that on $B_t(p,1)$, $k\in \mathbb{N}$,
$$Rm+\left[ Lt(scal_{g(t)}+3nk) +t^p+k \right]\mathcal{I}\in \mathfrak{C}.$$
for all $t\leq  T\wedge \tilde T$.
Here $\mathcal{I}$ denotes the constant curvature operator of scalar curvature $n(n-1)$.
\end{thm}
\begin{proof}
By a result in \cite[Lemma 8.1]{SimonTopping2016} and Shi's estimate (e.g. see \cite[Theorem 1.4]{CaoChenZhu2008} by shifting the time), we may assume the scalar curvature $scal_{g(t)} \geq -2nk$ and $|\nabla Rm_{g(t)}| \leq C_{n,a}t^{-3/2}$ on $B_t(p,\frac{3}{2}),\;t\in (0,T]\cap [0,\tilde T]$ by shrinking $\tilde T$. Let $\phi$ be a cutoff function on $[0,+\infty)$ such that $\phi=1$ on $[0,1]$, vanishes outside $[0,\frac{3}{2}]$ and satisfies
\begin{align}
|\phi'| \leq 100, \;\;\phi''\geq -100\phi.
\end{align}
Define $\Phi(x,t)=e^{-200mt}\phi^m\left( d_t(x,p)+c_n\sqrt{at}\right)$ where $m\in \mathbb{N}$ is an  positive integers to be specified later. Then, as in \eqref{dist-1}, we may assume that it satisfies
\begin{align}
\heat \Phi \leq -m\Phi
\end{align}
in the sense of barrier. We may assume it to be smooth when applying maximum principle as pointed out in the proof of Theorem \ref{pic}. Furthermore, we will assume $\tilde T\leq \frac{1}{1000m}$.

We consider
$$A=\Phi (Rm+k\mathcal{I})+\left[Lt(scal_{g(t)}+ 3nk)+\varphi(t)\right] \mathcal{I}$$
where $L$ is a large positive number and $\varphi(t)$ is a time function with $\varphi(0)\geq 0$ and $\varphi(t)\leq 1$. We will specify the choices later.

Clearly, $A(t)\in \mathfrak{C}$ on $B_t(p,\frac{3}{2}-c_n\sqrt{at})$ when $t$ is small. Without loss of generality we can assume $A(0)$ to be inside the interior of the cone at $t=0$ wherever $\Phi(x,0) >0$ ( if not, we can replace $k$ by $k+\delta$ and let $\delta \to 0$ eventually). Let $t_1>0$ be the first time such that $A(t)\in \mathfrak{C}$ on $B_t(p,\frac{3}{2}-c_n\sqrt{at})$ for all $t\in [0,t_1]$ and at $t=t_1$, there is a point $x_0\in \overline{B_t(p,\frac{3}{2}-c_n\sqrt{at})}$, $v\in \mathfrak{so}(n,\mathbb{C})$ with $|v|=1$, $v\in {S}$ so that $A(v,\bar v)=0$. We may assume that $t_1<1$, otherwise, it is done. Extend $v$ locally using parallel translation with respect to metric $g(t_1)$ and then extend it to spacetime such that $D_t v=0=\Delta v$ at $(p,t_1)$. Then at $(x_0,t_1)$, we have
\begin{equation}\label{equ--1}
\begin{split}
0&\geq \heat A(v,\bar v)\\
&=\Box \Phi \cdot (Rm + k\mathcal{I})_{v\bar v} +\Phi \Box R_{v\bar v}-2\langle \nabla R_{v\bar v} ,\nabla\Phi\rangle\\
&\quad +\varphi'+Lt|Ric|^2+L(scal_{g(t)}+3nk)\\
&\geq m\left[Lt(scal_{g(t)}+3nk)+\varphi \right]+2\Phi \cdot Q(Rm)_{v\bar v}+2\left(Lt\frac{\langle \nabla scal_{g(t)},\nabla \Phi\rangle}{\Phi}+R_{v\bar v}\frac{|\nabla\Phi|^2}{\Phi} \right)\\
&\quad+ \varphi' +Lt|Ric|^2 +L(scal_{g(t)}+3nk)\\
&\geq  m\left[Lt(scal_{g(t)}+3nk)+\varphi \right]+2\Phi \cdot Q(Rm)_{v\bar v}-\frac{m^2e^{400t}}{\Phi^\frac{2}{m}}\left[Lt(scal_{g(t)}+3nk)+\varphi \right]\\
&\quad- \frac{C_{n,a}mL}{\Phi^\frac{1}{m}t^\frac{1}{2}}-C_nm^2k+ \varphi' +Lt|Ric|^2 +L(scal_{g(t)}+3nk)\\
\end{split}
\end{equation}
Here, we have used $A_{v\bar v}=0$, $\tilde T<\frac{1}{1000m}$ and Shi's type estimate as stated above. On the other hand, by direct computation, we have
\begin{equation}\label{null-1}
\begin{split}
Q(A)_{v\bar v}&=\Phi^2 Q(Rm)_{v\bar v}+\Phi\left( \Phi k+\left[Lt(scal_{g(t)}+3nk)+\varphi \right]\right)(Ric\owedge I)_{v\bar v}\\
&\quad +(n-1)\left( \Phi k+\left[Lt(scal_{g(t)}+3nk)+\varphi \right]\right)^2\mathcal{I}_{v\bar v}.
\end{split}
\end{equation}

By \cite[Proposition 2.2]{CabezasBamlerWilking2017}, there is $\lambda>0$ such that the curvature type tensor $A$ satisfies
\begin{equation}\label{null-2}
\begin{split}
-\lambda\sqrt{Q(A)_{v\bar v}}&\leq \frac{1}{2}scal(A)-\left(Ric(A)\owedge I\right)_{v\bar v}\\
&=\frac{\Phi}{2}scal_{g(t)}-\Phi (Ric\owedge I)_{v\bar v}+\frac{1}{2}(n-1)(n-4) \left[ \Phi k+Lt(scal_{g(t)}+3nk)+\varphi\right]
\end{split}
\end{equation}

By combining \eqref{null-1} and \eqref{null-2}, we have
\begin{equation}
\begin{split}
&\quad Q(A)_{v\bar v}-\lambda\left( \Phi k+\left[Lt(scal_{g(t)}+3nk)+\varphi \right]\right) \sqrt{Q(A)_{v\bar v}} \\
&\leq \Phi^2 Q(Rm)_{v\bar v}+(n-1)\left( \Phi k+\left[Lt(scal_{g(t)}+3nk)+\varphi \right]\right)^2\\
&\quad +\frac{\Phi}{2}\left( \Phi k+\left[Lt(scal_{g(t)}+3nk)+\varphi \right]\right)scal_{g(t)} \\
&\quad +\frac{1}{2}(n-1)(n-4) \left( \Phi k+\left[Lt(scal_{g(t)}+3nk)+\varphi \right]\right)^2.
\end{split}
\end{equation}

And hence,
\begin{equation}\label{Null-3}
\begin{split}
2\Phi Q(Rm)_{v\bar v} &\geq -\Phi^{-1}\left( \frac{\lambda^2}{2}+(n-1)(n-2)\right)\left( \Phi k+\left[Lt(scal_{g(t)}+3nk)+\varphi \right]\right)^2\\
&\quad - \left( \Phi k+\left[Lt(scal_{g(t)}+3nk)+\varphi \right]\right) scal_{g(t)}
\end{split}
\end{equation}

Combines with  \eqref{equ--1} together, we get
\begin{equation}
\begin{split}
0&\geq m\left[Lt(scal_{g(t)}+2nk)+\varphi \right]-\frac{m^2e^{400t}}{\Phi^\frac{2}{m}}\left[Lt(scal_{g(t)}+3nk)+\varphi \right]\\
&\quad- \frac{C_{n,a}mL}{\Phi^\frac{1}{m}t^\frac{1}{2}}-C_nm^2k+ \varphi' +Lt|Ric|^2 +L(scal_{g(t)}+3nk)\\
&\quad - \left( \Phi k+\left[Lt(scal_{g(t)}+3nk)+\varphi \right]\right) scal_{g(t)}\\
&\quad -C_n\Phi ^{-1}\left( \Phi k+\left[Lt(scal_{g(t)}+3nk)+\varphi \right]\right)^2.
\end{split}
\end{equation}

On the other hand, we can use the fact that $A_{v\bar v}=0$ and the curvature assumption to deduce that
\begin{align}\label{zeroset}
\Phi \geq a^{-1}t\left[ Lt(scal_{g(t)}+3nk)+\varphi\right]\geq nLka^{-1}t^2 .
\end{align}

Hence, the second and the third term can be controlled by
\begin{equation}
\begin{split}
&\quad \frac{m^2e^{400t}}{\Phi^\frac{2}{m}}\left[Lt(scal_{g(t)}+3nk)+\varphi \right]+ \frac{C_{n,a}mL}{\Phi^\frac{1}{m}t^\frac{1}{2}}\\
&\leq \frac{m^2a^\frac{2}{m}}{t^\frac{2}{m}}\left[ Lt(scal_{g(t)}+3nk)+\varphi\right]^{1-\frac{2}{m}}+\frac{C_{n,a}a^\frac{1}{m}mL^{1-\frac{1}{m}}}{t^{\frac{1}{2}+\frac{2}{m}}}\\
&\leq \frac{C_0}{t^{\frac{1}{2}+\frac{2}{m}}}
\end{split}
\end{equation}
for some $C_0=C_0(n,m,L,k,a)>0$. The main obstacle is the last term. Write $Q=\left[Lt(scal_{g(t)}+3nk)+\varphi \right]$, then by using \eqref{zeroset}
\begin{equation}
\begin{split}
\Phi^{-1}(\Phi k+Q)^2
&\leq \frac{(1+\e)Q^2}{\Phi }+C_\e k^2\\
&\leq \frac{(1+\e)a}{t}\varphi +(1+\e)aL(scal_{g(t)}+3nk)+C_\e k^2.
\end{split}
\end{equation}

In conclusion, we have shown
\begin{equation}
\begin{split}
0&\geq m\left[Lt(scal_{g(t)}+2nk)+\varphi \right]-\frac{C_0}{t^{\frac{1}{2}+\frac{2}{m}}}\\
&\quad-C_{n,\e}m^2k+ \varphi' +Lt|Ric|^2 \\
&\quad +L(scal_{g(t)}+3nk)\left[ 1-(1+\e)a-t(scal_{g(t)}+3nk)- L^{-1}k-L^{-1}\varphi \right] \\
&\quad -\frac{(1+\e)a}{t}\varphi .
\end{split}
\end{equation}

Now if $a<\frac{1}{3\sqrt{n}}$, we may choose $\e>0$ such that $1>2(1+\e)a$. By taking $\varphi(t)=t^q,\;m,\;L$ so that  $a(1+\e)<q<\frac{1}{2}-\frac{2}{m}$, $1>(2+\e)a+L^{-1}(k+1)$. Then we see that $t_1 \geq \tilde T(n,a,k)$.

By shrinking $\tilde T$ further, we conclude that if $t\in [0,T]\cap [0,\tilde T]$, $x\in B_t(p,1)$, then
$$Rm+(k+Lt(scal_{g(t)}+3nk)+t^q) \mathcal{I} \in \mathfrak{C}.$$
\end{proof}


\section{Curvature estimates and pseudolocality}
When $\mathfrak{C}({S})$ is the cone corresponding to PIC1 or PIC2 condition, the situation is particularly interesting. In fact, it was proved in \cite{CabezasWilking2015} that any complete nonflat ancient solution of Ricci flow with
bounded curvature $Rm\in \mathfrak{C}_{PIC2}$ has
\[\lim_{r\to \infty} \frac{Vol B(r)}{r^n} = 0.\]
Moreover, in \cite{CabezasBamlerWilking2017}, they showed that in fact the any ancient solution with bounded curvature with $Rm\in \mathfrak{C}_{PIC1}$ must be in $\mathfrak{C}_{PIC2}$. In the following, we will focus on the cone $\mathfrak{C}_{PIC1}$. From now on, we will denote $\mathfrak{C}=\mathfrak{C}_{PIC1}$. Using the result on ancient solution with weakly PIC2, we can follow the argument in \cite{SimonTopping2016} to deduce the following estimates.

\begin{lma}\label{curvestimate}
For any $n,v_0,K>0$, there exists $\bar T(n,v_0,K)$, $C_0(n,v_0,K)>0$ such that the following holds: Suppose $(M^{n},g(t))$ is a Ricci flow for $t\in [0,T]$ and $p\in M$ such that $B_t(p,r)\subset\subset M$ for each $t\in[0,T]$. If
$$V_{g_0}(p,r)\geq v_0r^{n}\quad\text{and}\;\;Rm(g(t)) +Kr^{-2}\in\mathfrak{C}\;\;\text{on}\;\;B_t(p,r),\;t\in [0,T]$$
Then for all $t\in (0,T]\cap (0,\bar T\cdot r^2]$,
$$ |Rm|(x,t)\leq \frac{C_0}{t}\quad\text{on}\;\;B_t(p, r/8).$$
Moreover the injectivity radius satisfies
$$inj_{g(t)}(x)\geq \sqrt{C_0^{-1}t}.$$
\end{lma}
\begin{proof}
The proof is identical to the proof of Lemma 2.1 in \cite{SimonTopping2016} except that we use Lemma 4.2 in \cite{CabezasBamlerWilking2017} to draw a contradiction. The injectivity radius lower bound follows from a result in \cite{CheegerGromovTaylor1982} together with Lemma 2.3 in \cite{SimonTopping2016}.
\end{proof}

Using Lemma \ref{curvestimate} and Theorem \ref{pic}, the method in \cite{SimonTopping2016} can be carried over to give the following pseudolocality result.

\begin{thm}\label{pic1pse}Let $(M^n,g(t))$ be a complete solution of Ricci flow on  $M\times  [0,T]$. Let $p\in M$ and $r>0$.
  Suppose
  \begin{enumerate}
    \item [(i)] $Rm_{g_0}\in \mathfrak{C}$ on $B_{g_0}(p,8r)$;
    \item [(ii)] $V_0(x,ar)\ge v_0(ar)^{n}>0$ for all $x\in B_{g_0}(p,4r)$ and $a\leq 2$;
    \item [(iii)] $\sup_{M\times(\tau, T]}|Rm(g(t))|<\infty$ for all $\tau>0$.
  \end{enumerate}
     Then there is $C_0(n,v_0),\;\tilde T(n,v_0)>0$ such that
$$t|\Rm(g(t))|\le C_0$$  on $B_t(p,r)$ for all $t\in (0,T]\cap (0,\tilde Tr^2]$. 
\end{thm}
\begin{proof}
The proof is identical to the proof of Theorem 1.1 in \cite{SimonTopping2016}. We here only point out the main difference. Under our curvature assumption, the curvature estimate \cite[Lemma 2.1]{SimonTopping2016} can be replaced by Lemma \ref{curvestimate} here. The persistence of lower bound have been obtained in Theorem \ref{pic} if it is initially nonnegative. Moreover, in this proof the boundedness of curvature is only used when we apply pseudolocality at some positive time $t_0>0$, therefore we only need the curvature to be bounded for positive time.

\end{proof}

We next consider the case when the curvature of the initial metric $g_0$, $Rm(g_0)+k\mathcal{I} \in\mathfrak{C}_{PIC1}$ on some open set. Under a stronger assumption on the volume of geodesic balls, we have the following pseudolocality result even though initially $Rm(g_0)$ is not inside the cone.
\begin{thm}\label{pse-almost}For any $\e>0$ and $n\in\mathbb{N}$, there is $\delta(n,\e),\tilde T(n,\e)>0$ such that the following is true. Let $(M^n,g(t))$ be a complete solution of Ricci flow on  $M\times  [0,T]$, $p\in M$. Suppose the following is true.
  \begin{enumerate}
    \item [(i)] $Rm_{g_0}+k\mathcal{I}\in \mathfrak{C}$ on $B_{g_0}(p,2)$ for some $B>0$;
    \item [(ii)] $V_0(x,r)\ge (1-\delta)\omega_n r^{n}>0$ for all $x\in B_{g_0}(p,2)$ and $r\leq 2$;
    \item [(iii)] $\sup_{M\times(\tau, T]}|Rm(g(t))|<\infty$ for all $\tau>0$.
  \end{enumerate}
     Then for any $x\in B_t(p,\frac{1}{8})$, $t\in (0,T]\cap (0,\tilde T]$,
$$t|\Rm(g(t))|\le \e.$$
\end{thm}

We first show that if the geodesic ball is arbitrarily close to Euclidean, then the conclusion in Lemma \ref{curvestimate} can be strengthen to a arbitrary small upper bound.
\begin{lma}\label{almost-curvestimate}
For $\e,n>0$, there is $\tilde T,\delta,K>0$ such that if $(M,g(t))$ is a solution to the Ricci flow on $M\times [0,T]$, $p\in M$ satisfying
\begin{enumerate}
\item $B_t(p,1)\subset\subset M$ for $t\in [0,T]$;
\item $Vol_{g_0}(B_{g_0}(p,1)) \geq (1-\delta)\omega_n$;
\item $Rm_{g(t)}+K \mathcal{I} \in \mathfrak{C}_{PIC1}$ on $B_t(p,1)$, $t\in [0,T]$
\end{enumerate}
Then on $B_t(p,\frac{1}{2})$, $t\in [0,T]\cap [0,\tilde T]$, $$|Rm(x,t)|\leq \frac{\e}{t}.$$
\end{lma}
\begin{proof}
Suppose not, there is $\e_0>0$, $T_i,K_i,\delta_i\rightarrow 0$ and a sequence of Ricci flow $(M_i,g_i(t))$ corresponding to $T_i,K_i,\delta_i$ defined on $ [0,T_i]$ and there is a $p_i\in M$ satisfying
\begin{enumerate}
\item $Vol_{g_i(0)}(p_i,1)\geq (1-\delta_i) \omega_n$;
\item $Rm_{g_i(t)}+K_i \mathcal{I} \in \mathfrak{C}_{PIC1}$ on $B_{g_i(0)}(p_i,1)$.
\end{enumerate}
But there is $t_i\in [0,T_i]$ such that for all $t\in [0,t_i)$, $x\in B_{g_i(t)}(p_i,\frac{1}{2})$,
$$|Rm(g_i(t))|< \frac{\e}{t}.$$
And $|Rm(g_i(t))|(x_i,t_i)=\e t_i^{-1}$ for some $x_i\in \overline{B_{g_i(t_i))}(p_i,\frac{1}{2})}$.

By Corollary 6.2 in \cite{Simon2012}, we may choose $T_i$ small depending on $n$ and $\delta_i$ such that for all $t\in [0,t_i)\subset [0,T_i]$,
$$Vol_{g_i(t)}(B_{g_i(t)}(p_i,1))\geq (1-2\delta_i)\omega_n.$$





By Lemma 5.1 in \cite{SimonTopping2016}, for sufficiently large $i$, we can find $\bar t_i\in (0,t_i]$, $\bar x_i \in B_{g_i(\bar t_i)}(p_i,\frac{1}{2}+\frac{1}{2}\beta_n\sqrt{\e_0 \bar t_i})$ such that
$$|Rm|_{g_i(t)}(x)\leq 4|Rm|_{g_i(\bar t_i)}(\bar x_i)=4Q_i$$
whenever $d_{g_i(\bar t_i)}(x,\bar x_i)< 8^{-1}\beta_n \e_0 Q_i^{-1/2}$ and $\bar t_i-8^{-1} \e_0 Q_i^{-1}\leq t\leq \bar t_i$ where $Q_i\geq \e_0 \bar t_i^{-1}$. By volume comparison, for all $0<r< \frac{1}{2}$, for sufficiently large $i$,
\begin{align}\label{vol}
\frac{Vol_{g_i(\bar t_i)}(B_{g_i(\bar t_i)}(\bar x_i,r ))}{\omega_n r^n}\geq (1-3\delta_i)
\end{align}
if we choose $K_i$ small enough depending only on $\delta_i$ and $n$.

Consider $\tilde g_i(t)=Q_ig_i(\bar t_i+Q_i^{-1} t)$ $t\in [-\frac{\e_0}{8},0]$.
The rescaled Ricci flow satisfies
$$Rm(\tilde g_i(t))+K_i Q_i^{-1}\in \mathfrak{C}\quad\text{and}\quad|Rm|_{\tilde g_i(t)}\leq 4 \quad\text{on}\;\;B_{\tilde g(0)}(\bar x_i,\frac{1}{8}\beta \e_0),\,t\in [-\frac{\e_0}{8},0]$$
and $|Rm|_{\tilde g_k(0)}(\bar x_k)=1$. Moreover, by result in \cite{CheegerGromovTaylor1982}, we have uniform injectivity radius lower bound on $\tilde g_i(0)$ at $\bar x_i$ due to (\ref{vol}).

By local Hamilton compactness \cite[Theorem 3.16]{RicciFlowBook2}, we have a limiting solution $g_\infty(t)$ defined on $B_\infty\times [-\frac{\e_0}{8},0]$ which is non-flat, has Euclidean volume growth and has non-negative Ricci curvature. But this is impossible by volume comparison. This completes the proof.
\end{proof}

\begin{proof}[Proof of Theorem \ref{pse-almost}]
The proof is identical to that in Theorem \ref{pic1pse} as well as \cite[Theorem 1.1]{SimonTopping2016}. Here we only point out the necessary modifications. By scaling, we may assume $k$ to be small so that we can apply Lemma \ref{almost-curvestimate}. 
As in \cite[Page 27]{SimonTopping2016}, we only need to show that the case 2 in \cite[Lemma 5.1]{SimonTopping2016} is impossible with some choice of $c_0$. Now we can replace \cite[Lemma 6.1]{SimonTopping2016} by means of  \cite[Theorem 6.2]{SimonTopping2016} and Lemma \ref{almost-curvestimate} and conclude a curvature bound $at^{-1}$ where $a$ can be arbitrarily small depending on how the geodesic ball is close to a Euclidean one. Moreover, we have local persistence of curvature lower bound if we choose $c_0=\frac{1}{3\sqrt{n}}$. Then if the geodesic ball is too close to the Euclidean ball, we have $a<c_0$ yielding a contradiction at the centre of ball as in \cite[Page 28]{SimonTopping2016}.
\end{proof}

Suppose on a geodesic ball the sectional curvature is bounded, and the volume has a lower bound, then under a certain smaller scale geodesic balls have almost Euclidean volume (see for example \cite{Lu2010}), hence Theorem \ref{pse-almost} can be applied. Then Theorem 3.1 of \cite{chen2009} yields the following

\begin{cor}\label{local-doubling-estimate}
There exists $\sigma(n,v_0)$ and $\Gamma(n,v_0)$ such that for any complete smooth Ricci flow solution with $\sup_{M \times [\tau, T]}|Rm| < \infty$, $\forall\tau >0$, suppose $|Rm|(x, 0) \leq 1$ for all $x \in B_{g_0}(p, 1)$, and $Vol_{g_0}B_{g_0}(p, 1) \geq v_0$, then we have $|Rm|(x,t)\leq \Gamma$ for all $x\in B_{g(t)}(p, \sigma)$ and $t\in[0, \sigma^2 \wedge T]$.
\end{cor}


\section{Existence of Ricci flow}
 In this section, we will show how a short-time solution of Ricci flow can be constructed using local control Theorem \ref{pic}, Lemma \ref{curvestimate} and the method of \cite{Hochard2016} and \cite{SimonTopping2016}. For instance, if the initial metric is weakly PIC1 and is non-collapsed, then there is a short time solution of Ricci flow starting from such a metric.  The strategy here actually works as long as the curvature cone $\mathfrak{C}({S})$ satisfies the followings.
\begin{enumerate}
\item $Q(Rm)_{v\bar v}\geq 0$ for any $Rm\in \partial\mathfrak{C}$ and $v\in {S}$ with $Rm(v,\bar v)=0$;
\item Any nonflat ancient solution with bounded curvature with $Rm\in\mathfrak{C}$ has
\[\lim_{r\to \infty} \frac{Vol B(r)}{r^n} = 0;\]
\item $Ric(Rm)\geq 0$ for any $Rm\in \mathfrak{C}.$
\end{enumerate}

In particular, it is well-known that $\mathfrak{C}_{PIC1}, \mathfrak{C}_{PIC2}, \mathfrak{C}_{Rm}$ satisfy the above. From now on we let $\mathfrak{C}=\mathfrak{C}_{PIC1}$.  We would like to point out that Lai \cite{Lai2018} proved a stronger short-time existence result for Ricci flow on complete noncollapsed manifolds where $Rm(g)$ is almost weakly  PIC1.

\begin{thm}\label{RF}
  Let $(M,g_0)$ be a complete Riemannian manifold. Assume that
  $$V_{g_0}(x,1)\geq v_0 \;\;\text{for all}\;\; x\in M$$
  and $Rm_{g_0}$ is weakly $PIC1$. Then there is a complete solution of Ricci flow with $g(0)=g_0$ and $C_1(n,v_0)>0$ which satisfies
  $$\sup_M|Rm_{g(t)}|\leq \frac{C_1}{t}\;\quad \forall \;t\in (0,T(n,v_0)].$$
  Moreover, $Rm_{g(t)}\in\mathfrak{C}$.
\end{thm}
\begin{proof}
Let $R>>1$ be a fixed large number. Choose $\rho>0$ small enough such that
\begin{enumerate}
  \item $|Rm(g_0)|\leq \rho^{-2}$ on $B_{g_0}(p,R+1)$;
  \item $inj_{g_0}(x)\geq \rho$ for all $x\in B_{g_0}(p,R+1)$;
  \item $B_{g_0}(x,\rho)\subset B_{g_0}(p,R+1)$ for all $x\in B_{g_0}(p,R)$.
\end{enumerate}

Let $U=B_{g_0}(p,R+\rho)$, choose a conformal factor which is constant $1$ on $B_{g_0}(p,R)$ and blows up at $\partial U$ to turn $U$ into a complete manifold with bounded curvature, denoted as $\tilde{U}$. By Shi's solution of Ricci flow, there is a complete solution of Ricci flow on $X\times [0,t_0]$ where $X$ is the connected component of $\tilde U$ containing $B_{g_0}(p,R)$. Let $a(n,v)$ and $L(n,v)$ be some constants to be fixed later. By choosing $t_0 \leq 1$ small enough, we have a local solution of the Ricci flow with
$$|Rm(x,t)|\leq \frac{a}{t} \quad \text{on} \quad B_0(p,R)\times (0,t_0].$$

Note that $t_0$ depends on the initial metric on $U$ and can be very small, we extend the flow by the following inductive procedure.

\begin{claim}\label{improve}For any $x\in B_g(p,R-L\sqrt{t_0})$,
$$|Rm|\leq \frac{C_0}{t},\;\;inj_{g(t)}(x)\geq \sqrt{C_0^{-1}t},$$
where $C_0$ is the same as in Lemma \ref{curvestimate}.
\end{claim}
\begin{proof}[proof of claim]
Let $x\in B_0(p,R-L\sqrt{t_0})$. {By the shrinking ball lemma \ref{sbl}, we have for each $t\in[0, t_0]$
$$B_t(x, \sqrt{t_0}) \subset B_0(x,  \sqrt{t_0}+\beta \sqrt{a t_0}) \subset B_0(x,L \sqrt{t_0})\subset B_0(p,R).$$
{\bf Require: $L\geq  1+\beta \sqrt{a}$.}}

{By applying Theorem \ref{pic} on $B_t(x,\sqrt{t_0}/2)$, $t\in [0,t_0]$, with proper scaling, we have
$$Rm_{g(t)}+t\mathcal{I}\in\mathfrak{C}\quad\text{on}\;\; B_t(x,\sqrt{t_0}/16), \quad t \leq t_0 \wedge T(n,a) = t_0,$$
where we shrink $t_0$ again to make it small enough.

By volume comparison theorem,
$$V_0(x,\sqrt{t_0}/16)\geq \left(\sqrt{t_0}/16\right)^{2n} v.$$
Hence we may apply Lemma \ref{curvestimate} on $B_t(x,\sqrt{t_0}/16),\,t\in [0,t_0]$ to show that
$$|Rm|(x,t)\leq \frac{C_0}{t},\;\;inj_{g(t)}(x)\geq \sqrt{C_0 t^{-1}}.$$
Note that $C_0$ only depends on $n$ and $v$.}

\end{proof}

Repeat the conformal construction above with $U=B_0(p,R-L\sqrt{t_0})$ and $\rho=\sqrt{C_0^{-1}t_0}$. We may extend the Ricci flow solution to $t\in [0,(1+\mu)^2 t_0]$, where $(1+\mu)^2 =1+s_nC_0^{-1} $, on a slightly smaller set $\{x\in U: B_{t_0}(x,\rho)\subset\subset U\}$ in a way that for all $t\in [t_0,(1+\mu)^2 t_0]$
$$|Rm|\leq \frac{c_nC_0}{t_0} \leq \frac{a}{t}$$
where we have used the well-known doubling time estimate in Ricci flow, and $a$ is chosen to be $a=c_nC_0(1+\mu)$.

\begin{claim}
$\{x\in U: B_{t_0}(x,\rho)\subset\subset U\}\supset B_0(p,R-2L \sqrt{t_0}).$
\end{claim}
\begin{proof}[proof of claim]
For $x\in B_0(p,R-2L\sqrt{t_0})$,
$$ B_0(x,L\sqrt{t_0})\subset B_0(p,R-L\sqrt{t_0}).$$
By shrinking ball lemma \ref{sbl},
$$B_{t_0}\left(x,(L-\beta\sqrt{C_0})\sqrt{t_0}\right)\subset B_0(p,R-L\sqrt{t_0}).$$
The conclusion follows by our choice of $L(n,v)$.
\end{proof}
\noindent


Hence, we have a local solution of the Ricci flow with
$$|Rm|\leq \frac{a}{t}\quad\text{on}\;\;B_0(p,R-2L \sqrt{t_0})\times [0,t_1]$$
where $t_1=t_0(1+s_nC_0^{-1})=t_0(1+\mu)$.

Doing the above step inductively, we obtain a local solution to the Ricci flow on $B_g(p,s_k) \times [0,t_k]$ where
$$t_k=t_{k-1}(1+\mu)^2\quad\text{and}\quad s_k=R-2L \left(\sqrt{t_0}+\sqrt{t_1}+...+\sqrt{t_{k-1}} \right).$$

The process stops at the $k$-th step where $s_{k+1}<0$ or $t_{k+1}>\sigma_1(n,v)$, where $\sigma_1(n,v)$ is a positive constant determined by Lemma \ref{pic} and Theorem \ref{pic1pse}. If it is the first case, we restrict to eariler index $i<k$ where $s_i>R-1$ and $s_{i+1}\leq R-1$ but $t_i\leq \sigma_1$, then we can deduce
\begin{align*}
t_i \geq \frac{\mu^2}{4L^2(1+\mu)^4}=:\sigma_2(n,v).
\end{align*}
If it is the latter case where $t_k\leq \sigma_1$ and $t_{k+1}>\sigma_1$,
\begin{align*}
s_k&=R-2L \sqrt{t_k}\sum_{m=1}^k \frac{1}{(1+\mu)^m}\\
&\geq R-\frac{16(\mu+1)}{\mu}.
\end{align*}
In any cases, we have shown that $\exists \sigma(n,v), \delta(n,v),\,a(n,v)>0$ such that for all $R>\delta(n,v)$, there is a Ricci flow $g_R(t)$ defined on $B_g(p,R-\delta)$, $t\in [0,\sigma]$ such that for all $(x,t)\in B_g(p,R-\delta)\times (0,\sigma]$,
$$|Rm(g_R(t))|(x,t)\leq \frac{a(n,v)}{t}.$$

By letting $R\rightarrow \infty$ together with Shi's local estimate \cite{shi1989} and Chen's local estimate \cite{chen2009}, we can obtain a Ricci flow on $M\times [0,\sigma]$. The completeness follows from the shrinking ball lemma. $Rm(t) \in \mathfrak{C}$ follows from Corollary \ref{preservation of nonnegativity}.

\end{proof}

By rescaling the initial metric and theorem \ref{curvestimate}, the existence time will be infinity if the initial metric has maximal volume growth.
\begin{cor}\label{longtimesln}
Suppose $(M,g_0)$ is a complete Riemannian manifold with $Rm_{g_0}\in\mathfrak{C}$. If moreover $g_0$ has maximal volume growth, then there is a complete Ricci flow on $M\times [0,+\infty)$ satisfying
$$Rm(t) \in \mathfrak{C}, \quad |Rm|(t)\leq \frac{C_0}{t},\;\;\;inj_{g(t)}(x)\geq \sqrt{C_0t}$$
for some $C_0>0$.
\end{cor}

\section{Construction of the diffeomorphism in Theorem \ref{main theorem}}
\begin{proof}
[Proof of Theorem \ref{main theorem}]

Let $(M,g(t))$ be a complete Ricci flow, $t\in [0, \infty)$. Suppose that $Ric(t)\geq 0$ and for a fixed point $p$, the injectivity radius $inj_p(t) \geq r_t$, where $r_t$ is increasing $\to \infty$. We show the following elementary construction of a diffeomorphism between $M$ and $\mathbb{R}^n$. The main theorem \ref{main theorem} follows from this construction and Corollary \ref{longtimesln}.

Let $exp_t: T_pM \to M$ be the exponential map w.r.t $g(t)$. Identify $T_pM$ with $\mathbb{R}^n$ by fixing a time-independent basis. For simplicity, we denote
\[\tilde{B}_t(r) := \{x \in \mathbb{R}^n | x^T g(p,t) x < r^2\},\]
\[B_t(r) = B_{g(t)}(p, r),\]
then $exp_t (\tilde{B}_t(r)) = B_t(r)$ as long as $r < r_t$. We can choose a sequence of times $t_1 < t_2 < ...$ such that $r_{t_i} > i+1$. Since $Ric(t) \geq 0$, $\{B_{t_i}(i)\}$ is an exhaustion of $M$, and $\{\tilde{B}_{t_i}(i)\}$ is an exhaustion of $\mathbb{R}^n$. Denote $\mathcal{E}_t = exp_t^{-1}$, then $\mathcal{E}_t$ is defined on $B_t(r_t)$, however it does not converge to a diffeomorphism since the image of a fixed domain may shrink to a point as $t\to \infty$. To overcome this inconvenience, we deform $\mathcal{E}_{t_{i+1}}(B_{t_i}(i))$ back to $\tilde{B}_{t_i}(i)$ for each $i$, this is achieved by simply reversing the $t$-parametrized family of local diffeomorphisms on $\mathbb{R}^n$ induced by $\mathcal{E}_t$.

\textbf{Claim:} For each $i$, there is a diffeomorphism $\Phi_i : \mathbb{R}^n \to \mathbb{R}^n$, such that $\Phi_i(\mathcal{E}_{t_{i+1}}(B_{t_i}(i))) = \tilde{B}_{t_i}(i)$ and $\Phi_i$ is identity on $\mathbb{R}^n \backslash \tilde{B}_{t_{i+1}}(i+0.9)$. Moreover, we have $\Phi_i \mathcal{E}_{t_{i+1}} = \mathcal {E}_{t_i}$ on $B_{t_i}(i)$.
\begin{proof}[Proof of claim]
For simplicity let's take $i=1$. For each $t_2 > t>t_1$, observe that $\mathcal{E}_{t}$ embeds $B_{t_1}(1)$ into $\mathbb{R}^n$, and the image satisfies $\left( \mathcal{E}_{t}(B_{t_1}(1)) \cup \tilde{B}_{t_1}(1) \right) \\
\subset \tilde{B}_{t}(1) \subset \tilde{B}_{t_2}(1)$.

For an interior point $x$ in the domain of $\mathcal{E}_t$, $t_1< t < t_2$, let $V(x,t)$ be the tangent vector of the $t$-parameterized curve $\mathcal{E}_t(x) \subset \mathbb{R}^n$. Note that the vector field $V$ generates the family of local diffeomorphisms which evolves $\tilde{B}_{t_1}(1)= \mathcal{E}_{t_1}(B_{t_1}(1))$ into $\mathcal{E}_{t_2}(B_{t_1}(1))$. Now choose a smooth cutoff function $\psi$ such that $\psi =1$ on $\tilde{B}_{t_2}(1)$ and $\psi = 0$ on $\mathbb{R}^n \backslash \tilde{B}_{t_2}(1.9)$ for all $t_1 \leq t \leq t_2$. Let $\Psi_s$, $s\in [0,1]$ be the family of diffeomorphisms generated by $-\phi V(x,t_2 - (t_2-t_1)s)$ with $\Psi_0 = id$. Then $\Phi_1 := \Psi_1$ is the desired diffeomorphism.
\end{proof}

Inductively, we can find a sequence of diffeomorphisms $\Phi_i: \mathbb{R}^n \to \mathbb{R}^n$, such that
\[\Phi_{1}...\Phi_{i-1}\mathcal{E}_{t_i}\]
is diffeomorphic on $B_{t_i}(i)$, and it agrees with $\mathcal{E}_{t_j}$ on $B_{t_j}(j)$ for each $j< i$. We can let $i\to \infty$ to get a diffeomorphism from $M$ to $\mathbb{R}^n$.

\end{proof}


\begin{thebibliography}{1000}

\bibitem{RicciFlowBook2}B. Chow, S.-C. Chu, D. Glikenstein, C. Guenther, J. Isenberg, T. Ivey, D. Knopf, P. Lu,
F. Luo and L. Ni, {\sl The Ricci Flow: Techniques and Applications}, Part III. Mathematical Surveys and
Monographs, AMS, Providence, RI, 2008.

\bibitem{BrendleSchoen2009} Brendle, S.; Schoen, R., {\sl Manifolds with 1/4-pinched curvature are space forms.} J. Amer. Math. Soc. 22 (2009), no. 1, 287C307.


\bibitem{CabezasWilking2015}Cabezas-Rivas, E.; Wilking, B., {\sl How to produce a Ricci Flow via Cheeger-Gromoll exhaustion}, J. Eur. Math. Soc. (JEMS) 17 (2015), no. 12, 3153-3194.

\bibitem{CabezasBamlerWilking2017}Cabezas-Rivas, E.; R. Bamler; Wilking, B., The Ricci flow under almost non-negative curvature conditions, arXiv preprint arXiv:1707.03002 (2017).
\bibitem{ChauTamYu2011}Chau, A.; Tam, L.-F., Yu, C., {\sl Pseudo-locality for Ricci flow and applications}, Canad. J. Math. 63 (2011), no. 1, 55-85.


\bibitem{CaoChenZhu2008}Cao, H.-D.; Chen, B.-L; Zhu, X.-P., {\sl Recent developments on Hamilton's Ricci flow},
Surveys in differential geometry, Vol. XII, 47-112, Surv. Differ. Geom., XII, Int. Press,
Somerville, MA, 2008.

\bibitem{CheegerGromovTaylor1982}Cheeger, J.; Gromov, M.; Taylor, M., {\sl Finite propagation speed, kernel estimates for functions of the Laplace operator, and the geometry of complete Riemannian manifolds}, J. Differential Geom. 17 (1982), no. 1, 1553.
\bibitem{chen2009}Chen, B.-L., {\sl Strong uniqueness of the Ricci flow.} J. Differential Geometry 82 (2009) 363382.



\bibitem{LeeTam2017}Lee, M.-C.; Tam, L.-F., {\sl On existence and curvature estimates of Ricci flow}, arXiv preprint arXiv:1702.02667 (2017).
\bibitem{LeeTam2017-2}Lee, M.-C.; Tam, L.-F. {\sl Chern-ricci flow on noncompact complex manifolds}, arXiv preprint, arXiv:1708.00141, 2017.

\bibitem{Hochard2016} Hochard, R., {\sl Short-time existence of the Ricci  Ricci curvature bounded from below.} arXiv preprint, arXiv:1603.08726, 2016.
\bibitem{ST} Huang, S., and Tam, L.-F., {\sl K\" ahler-Ricci flow with unbounded curvature.} American Journal of Mathematics, Volume 140, Number 1, February 2018, pp. 189-220.

\bibitem{P} Perelman, G., {\sl The entropy formula for the Ricci flow and its geometric applications},
arXiv:math.DG/0211159

\bibitem{Lai2018} Lai, Y., {\sl Ricci flow under local almost non-negative curvature conditions}, preprint, arXiv:1804.08073.

\bibitem{Liu2013} Liu, G. {\sl 3-Manifolds with nonnegative Ricci curvature}, Invent. Math., 2013, 193: 367-375.

\bibitem{Lu2010} Lu, P. {\sl Local curvature bound in Ricci flow}, Geometry \& Topology , 2010 , 14 (2) :1095-1110

\bibitem{Menguy2000} Menguy, X., {\sl Noncollapsing examples with positive Ricci curvature and infinite topological type},  Geom. Funct. Anal. 2000, 10, no. 3, 600-627.
\bibitem{MicallefMoore1988}Micallef, M. J. and  Moore, J. D. {\sl Minimal two-spheres and the topology of manifolds with positive curvature on totally isotropic two-planes}, Annals of Mathematics, 1988, 127(1) :199-227.

\bibitem{Nguyen2010} Nguyen, H. T., {\sl Isotropic Curvature and the Ricci Flow}, International Mathematics Research Notices, 2010 , 23 (3) :536-558.

\bibitem{Simon2012}Simon, M., {\sl Ricci flow of non-collapsed three manifolds whose Ricci curvature is bounded from below},
J. Reine Angew. Math. 662 (2012), 59-94.

\bibitem{SimonTopping2016}Simon, M.; Topping, P. M.,  {\sl Local control on the geometry in 3D Ricci flow}, arXiv preprint arXiv:1611.06137 (2016).
\bibitem{shi1989}Shi, W.-X., {\sl Deforming the metric on complete Riemannian manifolds}, J. Differential Geom. 30 (1989), no. 1, 223301.
\bibitem{TianWang2015}Tian G.; Wang, B., {\sl On the structure of almost Einstein manifolds}, J. Amer. Math. Soc. 28 (2015), 1169-1209.
\bibitem{Wilking2013}Wilking, B., {\sl A Lie algebraic approach to Ricci flow invariant curvature conditions and Harnack inequalities}, J. Reine Angew. Math. 679 (2013), 223C247.


\bibitem{Brendle2018}Simon Brendle, {\sl Ricci flow with surgery in higher dimensions}, Annals of Mathematics 187 (2018), 263-299.
\end{thebibliography}
\end{document}